\documentclass[envcountsame,runningheads,notitlepage]{llncs}
\usepackage[a4paper, margin=1.1in]{geometry}
\usepackage{hyperref, amsmath,amssymb}
\usepackage{color}
\usepackage{tikz}
\title{The Scholz conjecture on addition chain is true for infinitely many integers with $\ell(2n)= \ell(n)$}
\titlerunning{The Scholz conjecture is true for $v(n) \leq 6$}
\date{}

\author{
  Amadou TALL\inst{1}
}

\institute{Departement de Mathématiques et Informatique\\Université Cheikh Anta Diop de Dakar\\
  \href{mailto:amadou7.tall@ucad.edu.sn}{amadou7.tall@ucad.edu.sn}
}  

\begin{document}
  \maketitle

\begin{abstract}
It is known that the Scholz conjecture on addition chains is true for all integers $n$ with $\ell(2n) = \ell(n)+1$. There exists infinitely many integers with $\ell(2n) \leq \ell(n)$ and we don't know if the conjecture still holds for them.
The conjecture is also proven to hold for integers $n$ with $v(n) \leq 5$ and for infinitely many integers with $v(n)=6$. There is no specific results on integers with $v(n)=7$. In \cite{thurber}, an infinite list of integers satisfying $\ell(n) = \ell(2n)$ and $v(n) = 7$ is given. In this paper, we prove that the conjecture holds for all of them.  
\end{abstract}

\keywords{addition chain, exponentiation, Scholz conjecture, scalar multiplication}

{\bf{Mathematics Subject Classification (MSC 2020)}} 11Y55, 11Y16

\section{Introduction}

Let $n$ be a positive integer. The problem of finding a minimal addition chain for $n$ is quite interesting. Addition chains can give the fastest exponentiation methods. Knowing a good way to reach $n$ from $1$ leads to a method of computing $x^n$.

\begin{definition}
An addition chain for a positive integer $n$ is a set of integers $\{a_0=1<a_1<a_2<\ldots<a_r=n\}$
 such that every element $a_k$ can be written as sum $a_i + a_j$ of preceding elements of the set.
\end{definition}

\begin{definition}
We define $\ell(n)$ as the smallest $r$ for which there exists an addition chain $\{a_0=1<a_1<a_2<\ldots<a_r=n\}$ for $n$.\\
\end{definition}

\begin{definition}
Let $n$ be an integer. We define $v(n)$ as the number of ''1''s in its binary expansion. Let us also define by $\lambda(n) = \lceil \log_2(n) \rceil$.
\end{definition}

The problem of finding $\ell(n)$ for a given $n$ is known to be NP--complete. An integer $n$ can also have several distinct minimal addition chains.
One of the most efficient method is the so-called ''fast exponentiation'' which refers to the binary method. It is also called the ''double--and--add'' method. It is proven to be the fastest method for all integers with $v(n) \leq 3$. 
 It is proven that 
\begin{theorem}
Let $n$ be a positive integer. Then,
\begin{enumerate}
\item If $v(n) = 1$, meaning $n=2^a$ then $\ell(n) = a$
\item IF $v(n) = 2$, meaning $n=2^a + 2^b$ then $\ell(n) = a +1$
\item IF $v(n) = 3$, meaning $n=2^a + 2^b + 2^c$ then $\ell(n) = a +2$
\end{enumerate}
\end{theorem} 
It become interesting to look at techniques based on the binary expansion of $n$. 
If $v(n) = 4$, then $n=2^a + 2^b + 2^c + 2^d$ and $\ell(n) \in \{a+2,a+3\}$. And it is the same case for $v(n)= 5$ where $\ell(n) \in \{a+3,a+4,a+5\}$.  \\
In \cite{thurber}, Thurber has been able to prove that there are integers with $v(n) \geq 6$ and $\ell(n) = a+4$. \\
It seems to be difficult to characterize the integers based on their binary representation.  In \cite{neill1}, Neill Clift manage to list all integers having $4$ or $5$ small steps in their minimal addition chains, meaning $\ell(n) = a+4$ or $\ell(n) = a+5$. 

The Scholz conjecture give a bound on the length of minimal addition chains for integers with only $1$s in their binary representation. In 1937, it was stated as follows:

\begin{conjecture}
Let $n$ be a positive integer. We have 
$$
\ell(2^n-1) \leq \ell(n) + n -1.
$$
\end{conjecture}

\medskip

Let us define the notion of short addition chain, which is not necessarily minimal as follows
\begin{definition}
Let $n$ be a positive integer, an addition chain for $2^n-1$ is called a {\it short addition chain}
 if its length is $\ell(n) + n -1$.
\end{definition}

In \cite{knuth}, it is proven to hold for $n \leq 16$. Later, Thurber \cite{thurber1} prove that it holds for $n \leq 32$. Aiello and Subbaru \cite{aiello} proved that it is true for all integers with $v(n) =1$.  It gains interested and have been proven to hold for $v(n) \leq 5$.

Thanks to Hatem \cite{hatem}, It is also true for $v(n) =6$ with $\ell(n) = \lambda(n) + 3$ and $\ell(n) = \lambda(n) + 5$.  \\
In 2005, Neill Clift \cite{neill2} confirmed that the Scholz conjecture is true for $n < 5784689$, the first non-hansen number. No results is known on integers with $v(n) = 7$ and $\ell(n) = \lambda(n) + 4$.

\medskip

Now, let us look at the product of integers. Thanks to the factor method, we can see that
$$
\ell(mn) \leq \ell(m) + \ell(n), \forall m, \ n
$$
We are tempted to believe that $\ell(2n) = \ell(n)+1$ and 
it is easy to prove the following:
\begin{lemma}
It the Scholz conjecture hold for $n$, and $\ell(2n) = \ell(n)+1$, then it holds for $2n$
\end{lemma}
\begin{proof}
Let $n_0 = 2n$ be another positive integer, we have 
$$
2^{n_0} - 1 = (2^n-1)(2^n+1), 
$$
using the factor method, we can deduce a chain for $2^{n_0} - 1$ of length 
$$
\ell(n) + n - 1 + n+1 = \ell(n) + 2n = \ell(n_0) + n_0 -1.
$$
The chain is 
$$
\mathcal{C} = \{1,~2,~\ldots,~2^n-1,~2(2^n-1),~2^2(2^n-1), \ldots,~2^n(2^n-1),~2^n(2^n-1) + (2^n-1) = 2^{2n}-1\}
$$
\end{proof}
However, it has also been proven that there are infinitely many integers with $\ell(2n) \leq \ell(n)$. Thurber \cite{thurber} has listed a group of integers with $v(n) = 7$, $\ell(n) = \lambda(n) + 4$ and $\ell(2n) = \ell(n)$.
In this paper, we prove that the Scholz conjecture is true for his list.

\section{Our contribution}

\subsection{Tools to prove our main results}

Let us give a way to construct addition chains for $2^n-1$ based on addition chains for $n$. We will see later that it can help to get short addition chains.

	\begin{lemma}
	If $n = 2A$ for some $A$, then we can construct a chain for $2^n-1$ by adding $A+1$ steps to a chain for $2^A-1$.
	\end{lemma}
	\begin{proof}
	$$
	2^n-1 = 2^{2A}-1 = (2^A-1)(2^A+1)
	$$
	Using the factor method, we can deduce a chain for $2^n-1$ with respect to the theorem as follows
	$$
	\mathcal{C} = \{1,~2, \cdots, (2^A-1),~2(2^A-1),2^2(2^A-1),\cdots,2^A(2^A-1),2^A(2^A-1)+(2^A-1)=2^n-1\}
	$$
	\end{proof}
	
	\begin{lemma}
	Let $n=A+B$ be an integer with $A$ and $B$ appearing in an addition chain for $n$ ($A >B$).. Then, we can construct an addition chain for $2^n-1$ by adding $B+1$ steps to a chain for $2^A-1$ which contains $2^B -1$.
	\end{lemma}
	\begin{proof}
	
	$$
	n = A+B \Rightarrow 2^n-1 = 2^{A+B} - 1 = 2^B(2^A-1) + (2^B-1)
	$$
	So, if we have an addition chain for $2^A-1$ which contains $2^B-1$, it easy to construct a chain for $2^n-1$ as follows
	$$
	\mathbb{C}_n =\{1,2, \ldots, 2^B-1, \ldots, 2^A-1, 2(2^A-1),\ldots, 2^B(2^A-1),n=2^B(2^A-1)+(2^B-1)\}.
	$$
	\end{proof}
	
Let us illustrate it with an example.\\
\begin{example}
 Let $n=11$ and $\mathcal{C} = {1,2,3,5,10,11}$ be a chain for $11$. We will deduce a chain for $2^{11}-1$ as follows
\begin{enumerate}
\item $1$ is the first element of the chain
\item $2 = 2 \times 1$ is in the chain so we will add $2$ and $2^2-1 = 3= 2+1$
\item $3 = 2 +1$, so we need a chain for $2^2-1=3$ which contains $2^1 -1 = 1$, we add to the chain $2 \times 3=6$ and $2 \times 3 + 1 = 7$ 
\item  $5 = 3 + 2$, we need a chain for $2^3-1 = 7$ which contains $2^2-1 = 3$, we add $2 \times 7$, $2^2 \times 7$ and last $2^2 \times 7 + 3 = 31$
\item and so on
\item The chain for $2^{11}- 1$ is then 
$$
\mathcal{C}= \{1,2,3=2^2-1,6,7=2^{3}-1,14,28,31=2^{5}-1,62,124,248,496,992,1223=2^{10}-1,2446,2447=2^{11}-1\}
$$
\end{enumerate}	

\end{example}

\section{Our main results}

Let us start with:

\begin{lemma}
Let $m$ and $k$ be two positive integers with $k\geq 3$.
Let $c_1 = 101\underbrace{0\cdots0}_{m}11 = 5 \cdot 2^{m+2}+3 $ and $c_2=11\underbrace{0\cdots0}_{m}1 = 3\cdot 2^{m+1}+1$ be two integers. Then,
$$
\ell(c_1) = m+6, \quad \text{ and } \ell(c_2) = m+4
$$
and we can construct a chain for $c_1$ of length $m+7$ which contains $c_2$.
\end{lemma}
\begin{proof}
It is easy to see that $v(c_1) = v(c_2) = 4$ and \cite{knuth} prove that $\ell(c_1) = \lambda(c_1) + 2 = m+6$. Similarly for $c_2$.\\
 
Now, One can see that $c_1 = 3c_2+2^{m+1}$, so a chain can be constructed as follows
$$
\mathcal{C} = \{1,~2,~\ldots,~2^{m+1},~2\cdot 2^{m+1},~3 \cdot 2^{m+1},~c_2,~2c_2,~3c_2 = 2c_2+c_2,~3c_2+2^{m+1}\}
$$
and $\ell(\mathcal{C}) = m+7$.
\end{proof}
\begin{lemma}
We can construct a chain for $2^{c_1}-1$ that contains $2^{c_2}-1$ of length $\ell(c_1) + c_1 = c_1 + m + 6$.
\end{lemma}
We can see that the chain doesn't respect the Scholz bound but it is enough for our proof. 
\begin{proof}
We know that $c_1 = 3c_2+2^{m+1}$, so 
\begin{align}
2^{c_1}-1 &= 2^{3c_2+2^{m+1}} - 1 \\
 &= 2^{2^{m+1}}(2^{3c_2}-1)+(2^{2^{m+1}}-1)\\
  &= 2^{2^{m+1}}(2^{c_2}(2^{2c_2}-1)+(2^{c_2}-1))+(2^{2^{m+1}}-1)\\
    &= 2^{2^{m+1}}(2^{c_2}((2^{c_2}-1)(2^{c_2}+1))+(2^{c_2}-1))+(2^{2^{m+1}}-1)
\end{align}
Then, we can construct a chain for $2^{c_1}-1$ which contains $2^{c_2}-1$ and $2^{2^{m+1}}-1$ as follows
\begin{enumerate}
\item Start by a chain for $2^{c_2}-1$ which contains $2^{2^{m+1}}-1$ using the chain $c_2$ 
\item Use the factor method to get the chain for $(2^{c_2}-1)(2^{c_2}-1) = 2^{2(c_2)}-1 $
\item Add $c_2$ doubling to get $2^{c_2}(2^{2(c_2)}-1) = 2^{3c_2}-1$
\item Add $2^{m+1}$ doubling to reach $2^{2^{m+1}}(2^{3c_2}-1)$
\item Add $2^{2^{m+1}} - 1$
\end{enumerate}
The total length is $\ell(2^{c_2}-1)+c_2+(c_2+1)+1+2^{m+1}+1 = c_1 + m+6$.
\end{proof}

Our first result is:
\begin{theorem}
Let $m$ and $k$ be two positive integers with $k\geq 3$. The Scholz conjecture on addition chain is true for all integers
of the form 
$$
n=101\underbrace{0\cdots0}_{m}11\underbrace{0\cdots0}_{k}11\underbrace{0\cdots0}_{m}1 = c_1\cdot 2^{2m+k+3}+c_2,
$$
with $c_1 = 101\underbrace{0\cdots0}_{m}11 = 5 \cdot 2^{m+2}+3 $ and $c_2=11\underbrace{0\cdots0}_{m}1 = 3\cdot 2^{m+1}+1$.
\end{theorem}

\begin{proof}
We know that 
\begin{align}
2^n-1 &= 2^{c_1\cdot 2^{2m+k+3}+c_2}-1 \\
 &= 2^{c_2}(2^{c_1\cdot 2^{2m+k+3}}-1)+(2^{c_2}-1) \\
 &= 2^{c_2}((2^{c_1}-1)(2^{c_1}+1)(2^{2c_1}+1)(2^{2^2c_1}+1)\cdots (2^{2^{2m+k+2}c_1}+1))+(2^{c_2}-1)
\end{align}
And we have a chain for $2^{c_1}-1$ which contains $2^{c_2}-1$. The following is an addition chain for  $2^n-1$ 
$$
\mathcal{C} = \{1,~2,~\ldots,~(2^{c_2}-1),\ldots,(2^{c_1}-1),~\ldots,~(2^{2c_1}-1)=(2^{c_1}-1)(2^{c_1}+1),
$$
$$
\ldots,~(2^{2^{2m+k+3}c_1}-1),~2(2^{2^{2m+k+3}c_1}-1),~\ldots,~2^{c_2}(2^{2^{2m+k+3}c_1}-1),~n\}
$$
its length is 
$$
(c_1+m+6) + c_2+(2m+k+3)+c_1(2^{1m+k+3}-1)+1 = n + 2m+k+10= \ell(n)+n-1
$$

Some explanations can be found below:
\begin{enumerate}
\item $c_1+1$ steps to go from $2^{c_1}-1$ to $2^{2^2c_1}-1 = (2^{c_1}-1)(2^{c_1}+1)$
\item $2c_1+1$ steps to go from $2^{2c_1}-1$ to $2^{2^2c_1}-1 = (2^{2c_1}-1)(2^{2c_1}+1)$
\item $2^2c_1+1$ steps to go from $2^{2c_1}-1$ to $2^{2^{2^2}c_1}-1 = (2^{2^2c_1}-1)(2^{2^2c_1}+1)$
\item and so on 
\item $2^{2m+k+2}c_1+1$ steps to go from $2^{2^{2m+k+2}c_1}-1$ to $2^{2^{2m+k+3}c_1}-1 = (2^{2^{2m+k+2}c_1}-1)(2^{2^{2m+k+2}c_1}+1)$
\end{enumerate}
\end{proof}
Our next result will be to prove that the Scholz conjecture is also true for $2n$.

\begin{theorem}
Let $n$ be defined as in the previous theorem. The Scholz conjecture on addition chain is true for 
$2n$.
\end{theorem}
\begin{proof}
Let us remind that 
$$
2n = 101\underbrace{0\cdots0}_{m}11\underbrace{0\cdots0}_{k}11\underbrace{0\cdots0}_{m}10 = (2^{m+4}+2^{m+2}+2+1) \cdot (2^{m+k+4})+(2^{m+3}+2^{m+2}+2),
$$
and let us denote by $c_3= (2^{m+4}+2^{m+2}+2+1)$ and $c_4=2^{m+3}+2^{m+2}+2$. 
A minimal addition chain for $c_3$ which contains $c_4$ is 
$$
\mathcal{C} = \{1,~2,\ldots,~2^{m+2},~,~2^{m+2}+1,,~2^{m+3}+1,,~2^{m+3}+2^{m+2}+2,~2^{m+4}+2^{m+2}+2+1\}
$$
meaning that we can have a short addition chain for $2^{c_3}-1$ which contains $2^{c_4}-1$. \\

\medskip

An addition chain for $2^n-1$ can be obtained with the following expression,  
\begin{align}
2^{2n}-1 &= 2^{(2^{m+4}+2^{m+2}+2+1) \cdot (2^{m+k+4})+(2^{m+3}+2^{m+2}+2)}-1 = 2^{c_3 \cdot (2^{m+k+4})+ c_4 }-1 \\
 &= 2^{c_4}(2^{c_3 \cdot (2^{m+k+4})}-1)+(2^{c_4}-1) \\
  &= 2^{c_4}((2^{c_3}-1)(2^{c_3}+1)(2^{2c_3}+1)\cdots((2^{2^{m+k+3}c_3}+1)))+(2^{c_4}-1) \\
\end{align}

Similar techniques than in the previous result can be applied to get an addition chain for $2^{2n}-1$ of length $(\ell(c_3) + c_3-1)+c_4+(m+k+4)+c_3(2^{m+k+4}-1) = 2n + 2m+k+10= \ell(2n)+2n-1$.
\end{proof}

\begin{theorem}
The Scholz conjecture on addition chains is true for infinitely many integers $n$ with $\ell(2n) = \ell(n)$.
\end{theorem}
\begin{proof}
Let $m$ and $k$ be two positive integers with $k\geq 3$. \\
Let $n=101\underbrace{0\cdots0}_{m}11\underbrace{0\cdots0}_{k}11\underbrace{0\cdots0}_{m}1$ be a positive integer. 
We have proven that the Scholz conjecture is true for both $n$ and $2n$. 
\end{proof}

\section{Conclusion}
We have proved that the Scholz conjecture on addition chains is true for infinitely many integers $n$ with $\ell(2n) = \ell(n)$. It is still an open problem in general. Also, we know that there are infinitely many integers $m$ and $n$ that satisfy $\ell(mn) \leq \ell(m)$, one can investigate their behavior with the conjecture.

\section*{Acknowledgments}
The author acknowledge the support of IHES. The work was completed during his research visit. The question arises during enjoyful discussions with Mike Bennett at UBC. This work was partially supported by a grant from the Simons Foundation throug IMU.\\

\end{document}